\documentclass[12pt]{amsart}
\usepackage{graphicx}
\usepackage{amssymb}
\usepackage{epstopdf}

\usepackage{amsmath}
\usepackage{amsthm}
\usepackage{a4wide}
\usepackage{pdfsync}
\usepackage{hyperref}
\usepackage{color}

\usepackage[pagewise,mathlines]{lineno}

\newtheorem{lemma}{Lemma}[section]
\newtheorem{theorem}{Theorem}[section]

\newcommand{\sgn}{\,{\rm sgn}}

\def\rr{\mathbb{R}}

\def\la{\lambda}
\def\cgk{C_{GK}}

\title[A nonlocal convection diffusion equation]{Long-time behavior for a  nonlocal convection diffusion equation}
\author[L. I. Ignat, T. I. Ignat ]{Liviu I. Ignat, Tatiana I. Ignat}
\date{}                                           

\address{L. I. Ignat
\hfill\break\indent Institute of Mathematics ``Simion Stoilow'' of the Romanian Academy\\
\hfill\break\indent  21 Calea Grivitei Street \\010702 Bucharest \\ Romania.}

 \email{{\tt
liviu.ignat@gmail.com}\hfill\break\indent  {\it Web page: }{\tt
http://www.imar.ro/\~\,lignat}}

\address{T. I. Ignat
\hfill\break\indent Research group of the project PN-II-RU-TE-2014-4-0657\\
\hfill\break\indent Institute of Mathematics ``Simion Stoilow'' of the Romanian Academy\\
\hfill\break\indent  21 Calea Grivitei Street \\010702 Bucharest \\ Romania 
}
 \email{\tt tatiana.ignat@gmail.com}

\begin{document}

\begin{abstract}
In this paper we consider a nonlocal viscous Burgers' equation and study the well-posedness and  asymptotic behaviour of its solutions. We prove that under the  smallness assumption on the initial data the solutions behave as the self similar profiles of the Burgers' equation with Dirac mass as the initial datum. The first term in the asymptotic expansion of the solutions is obtained by rescaling the solutions and proving the compactness of their trajectories. 
\end{abstract}

\maketitle

{\textit{Mathematics Subject Classification 2010}:}  76Rxx, 35B40, 45M05, 45G10.

{\textit{Key words}:} asymptotic behavior, nonlocal diffusion, nonlocal convection

\section{Introduction}

In this paper we analyze the following equation 
\begin{equation}\label{conv-diff}
\left\{
\begin{array}{rl}
\ u_t(t,x)&=\displaystyle\int _{\rr}K(x-y)(u(t,y)-u(t,x))dy\\
&\quad \quad  \quad \displaystyle+\int _{\rr} G(x-y)f\Big(\frac {u(t,y)+u(t,x)}2\Big)dy,t>0, x\in \rr,\\[10pt]
u(0)&=\varphi,
\end{array}
\right.
\end{equation}
in the particular case when $f(u)=u^2$. 
This model has been proposed in  \cite{MR2888353} as a regularization of the following nonlocal advection equation inspired by the peridynamic theory
\[
  \displaystyle u_t(t,x)=\int _{\rr} G(x-y)f\Big(\frac {u(t,y)+u(t,x)}2\Big)dy.
\]
The general model in \cite{MR2888353} assumes that $K$ is a nonnegative even function and $G$ is an odd function. We consider here kernels $K$ and $G$ that are integrable.  For simplicity we  assume that $K$ has mass one.
We will analyze the well-posedness of system \eqref{conv-diff} and the long time behaviour of its solutions. 
The results presented here hold under the assumption that kernel $K$ dominates $G$, i.e. for some positive constant $C=C_{GK}$ the following holds
\begin{equation}\label{hyp.1}
|G(x)|\leq \cgk K(x), \quad \forall \ x\in \rr.
\end{equation}
Using that function $G$ is an odd function and the nonlinearity is given by $f(u)=u^2$  equation \eqref{conv-diff} can be written as
\begin{equation}\label{conv-diff-1}
\left\{
\begin{array}{ll}
u_t=K\ast u-u+G\ast \frac{u^2}4 + (G\ast  u )\frac u2,& ,t>0,x\in \rr,\\[10pt]
u(0)=\varphi.
\end{array}
\right.
\end{equation}


In the particular case 
when $K(x)=e^{-|x|}/2$ and $G=K'$ the considered model is similar, excepting the last term, $(G\ast u)u$, to the BBM-Burgers equation than can be written in the form
\[
  u_t=K\ast u-u +G\ast (\alpha u+\beta u^2),
\]
for some real constants $\alpha$ and $\beta$.
In this case the decay of solutions has been studied under various assumptions on the regularity of the solutions,  i.e. $u(0)\in L^1(\rr)\cap H^2(\rr)$ in \cite{MR1012198}, small $L^1(\rr)\cap H^1(\rr)$ solutions in \cite{MR1727212}. A different approach with a better convergence to the asymptotic profile has been obtained in \cite{MR1847842}. However,  the asymptotic profile in \cite{MR1847842} is not explicit as the one obtained in \cite{MR1727212} or \cite{0762.35011}.

In this paper we first prove the global well-posedness of system \eqref{conv-diff} and the decay of its solutions. 

\begin{theorem}\label{global}For any $u_0\in  L^1(\rr)\cap L^\infty(\rr)$ there exists a unique local solution $u\in C([0,T_{max}], L^1(\rr)\cap L^\infty(\rr) )$ of equation \eqref{conv-diff}, a solution that conserves the mass of the initial data. Moreover,  under the smallness assumption on the initial data $\|\varphi\|_{ L^\infty(\rr)}< 1/\cgk$ the solution is global, preserves the sign of the initial data and satisfies
\begin{equation}\label{stability-l1}
\|u(t)\|_{L^1(\rr)}\leq \|\varphi\|_{L^1(\rr)},
\end{equation}
\begin{equation}\label{stability-infty}
\|u(t)\|_{L^\infty(\rr)}\leq \|\varphi\|_{L^\infty(\rr)}.
\end{equation}
\end{theorem}

\begin{theorem}\label{decay}Assume the conditions in Theorem \ref{global} and in addition $K\in L^1(\rr,1+x^2)$.  The solution of system \eqref{conv-diff} satisfies      
\begin{equation}\label{decay.2}
\|u(t)\|_{L^2(\rr)}\leq C(\|\varphi\|_{L^1(\rr)}, \|\varphi\|_{L^\infty(\rr)})t^{-\frac 14}, \quad \forall t>0. 
\end{equation}
Moreover, if the initial data satisfy  $\|u_0\|_{ L^\infty(\rr)}< 1/2\cgk$  the following estimate holds for any
$2\leq p<\infty$
\begin{equation}\label{decay.p}
\|u(t)\|_{L^p(\rr)}\leq C(\|\varphi\|_{L^1(\rr)}, \|\varphi\|_{L^\infty(\rr)})t^{-\frac 12(1-\frac 1p)}, \quad \forall t>0. 
\end{equation}
\end{theorem}

The above condition on the smallness of the $L^\infty(\rr)$-norm of the solution  is similar to the CFL-condition that appears in the study of the stability of the numerical approximations for conservation laws, see \cite[Ch. 3]{0768.35059}.
 This guarantees that the diffusive part controls the nonlinear convective term. The difference with the previous works on nonlocal convection-diffusion equations \cite{MR2356418,MR3190994,denisa} is that in this case the convective term 
 \[
  Tu=\int _{\rr} G(y-x)(\frac {u(t,y)+u(t,x)}2)^2dy
\]
does not satisfy the dissipative condition $(Tu,u)_{L^2(\rr)}\leq 0$. Thus, extra assumptions should be imposed to the initial data, so on the solutions, to control this term by the diffusive part.   

The condition on $K$ to belong to $L^1(\rr,1+x^2)$ and \eqref{hyp.1} imposes that $G\in L^1(\rr,1+x^2)$. This condition on $G$ does not appear in the previous works \cite{MR2356418,MR3190994,denisa} where $G\in L^1(\rr)$ is sufficient to obtain the decay of the solutions and $G\in L^1(\rr,1+|x|)$ to obtain
the first term in the asymptotic expansion of solutions. Regarding the assumption \eqref{hyp.1} we don't know whether there is blow-up in finite time when this condition does not hold. The same happens with the smallness assumption on the initial data. These questions remain to be analyzed.

%
%
%

We need to assume  $\|u_0\|_{ L^\infty(\rr)}< 1/(2\cgk)$   in order to prove the decay property  \eqref{decay.p} for large $p$. When werestrict the range of $p$ this condition can be slightly improved as in the case $p=2$. This condition can be merely technical and it is related with the best constant in some inequalities that we will use in the proof of Theorem \ref{decay}.

We now introduce the following quantities
$$A=\frac 12 \int _{\rr}K(z)z^2dz\quad \text{and}\quad B=\int _{\rr}G(z)zdz.$$
The asymptotic expansion of the solutions of system \eqref{conv-diff} will depend on the previous quantities $A$ and $B$.
The main result concerning the asymptotic expansion for the solutions of system \eqref{conv-diff} is the following one. 
\begin{theorem}\label{asimp}Let us assume that $K\in L^1(\rr,1+x^2)$ is positive in a neighborhood of the origin.
For any $\varphi\in L^1(\rr^d)\cap L^\infty(\rr^d)$ with $\|\varphi\|_{L^\infty(\rr)}\leq 1/(2\cgk)$  solution $u$ of system \eqref{conv-diff} satisfies
\begin{equation}\label{lim.t}
\lim _{t\rightarrow \infty} t^{\frac  d2(1-\frac 1p)}\|u(t) -U(t)\|_{L^p(\rr^d)}=0, \quad  1\leq p<\infty,
\end{equation}
where $U$ is the solution of the viscous Burgers' equation 
\begin{equation}\label{burgers}
\left\{
\begin{array}{ll}
U_t=A  U_{xx}- \frac {B}2(U^2)_x,&  t>0,x\in \rr^d,\\[10pt]
U(0)=m\delta _0,
\end{array}
\right.
\end{equation}
 and $m$ is the mass of the initial data $\varphi$.
\end{theorem}

Next, we say a few words about the above asymptotic profile $U$. 
The well-posedness of  system \eqref{burgers} has been analyzed in \cite{MR1233647} in the one-dimensional case and  in \cite{MR1266100} in the multi-dimensional case. Function $U$ satisfies $U(t,x)=t^{-1/2}f_m(x/\sqrt{t})$ where $f_m$ is the smooth solution with mass $m$ of  the equation
\[
  -A f_m''-\frac 12 x f_m'(x)=\frac 12 f_m -  \frac{B}2(f_m^2)'(x) \quad \text{in}\  \rr.
\]
It has been proved in \cite{MR1032626} that the profile $f_m$ is of constant sign and decays exponentially to zero as $|x|\rightarrow \infty$. The interested reader can consult \cite{MR1032626} for more properties of these profiles $f_m$.

One of the central properties in the analysis of Burgers' equation is the so-called Oleinik's estimate $u_x\leq 1/t$. In the nonlocal setting there are few results in this direction under the assumption that the convection is local $(|u|^{q-1}u)_x$, $q\in (1,2]$, see \cite{MR2138795, cazacu}. The existence of similar estimates, in local or nonlocal form, for the models where the convection is nonlocal is, as far the authors know, an open problem.

In this paper we consider the nonlinearity $f(u)=u^2$. The extension of the results in this paper to the case of more general nonlinearities of the form $f(u)=|u|^{q-1}u$, $q>1$ or the multi-dimensional case  remains to be investigated in the future.

The paper is organized as follows. In Section \ref{well} we prove Theorem \ref{global} and Theorem \ref{decay}. Section \ref{expansion} concerns the asymptotic expansion of the solutions of Theorem \ref{asimp}.

\section{Global well-posedness and the decay of the solutions}\label{well}

In this section we  prove the global existence of solutions of system \eqref{conv-diff-1} and  the decay of its solutions in the $L^p(\rr)$-norm with $1\leq p<\infty$. To simplify the presentation we will not make explicit the dependence of the functions on the arguments unless this is necessary. 

\begin{proof}[Proof of Theorem \ref{global}]
The local existence of the solutions is obtained by the classical Banach fix point theorem and we will omit it. The reader can consult \cite{MR2356418} for a similar model. The mass conservation follows by the obseration that $K$ has mass one and $G$ is an odd function. Theses properties guarantee that the two terms in the right hand side of \eqref{conv-diff} vanish when integrated on the whole space $\rr$.

We will now prove that under the smallness assumption of the $L^\infty(\rr)$-norm of the initial data the $L^1(\rr)$ and $L^\infty(\rr)$ norms of the solution  do not increase and in consequence they remain  controlled by the ones of the initial data. The global existence is then trivial.

\medskip

\textbf{Step I. Control of the $L^\infty(\rr)$-norm.}
 In the following we denote by $w^+=\max\{w,0\}$ and $w^-=\min\{w,0\}$ the positive and negative parts of function $w$.

Let us now prove that the $L^\infty(\rr)$-norm of the solution does not increase under the smallness assumption on the initial data.
 Set $m=\|\varphi\|_{L^\infty(\rr)}$. Let us choose $M>m$ where further conditions on $M$ will appear in the proof depending on $\cgk$. 
 Using that  $u\in C([0,T],L^\infty(\rr))$ there exists an interval $0\in I \subset [0,T]$ such that $\|u(t)\|_{L^\infty(\rr)}\leq M$ for $t\in I$. Under this assumption we prove
 that the positive part of $u(t)-m$ satisfies
 \[
  \frac {d}{dt}\int _{\rr} (u(t,x)-m)^+dx\leq 0, \quad \forall \,t\in I
\]
  and hence $u(t)\leq m$ for all $t\in I$.
Let us write $u=m+v$. Then $v(0)\leq 0$. Since $K$ has mass one and  $G$ is an odd function, then $v$ satisfies the following equation
\begin{align}
\label{eq.v}
v_t&=K\ast v-v+G\ast (\frac {m^2+2mv+v^2}4)+(v+m)G\ast \frac{v+m}2 \\
&=K\ast v-v + m G\ast v+ G\ast \frac{v^2}4 + (G\ast v)\frac{v}2.
\nonumber
\end{align}
Multiplying \eqref{eq.v} by $\sgn (v^+)$ and integrating in the space variable  we obtain 
 \begin{align*}
\frac {d}{dt}\int _{\rr} v^+(t,x)dx&=-\int _\rr{v\sgn (v^+)}dx+\int _{\rr}\Big(K\ast v+ m G\ast v+ G\ast \frac{v^2}4 + (G\ast v) \frac v 2\Big)\sgn (v^+)dx\\
&=-\int _{\rr}v^+ dx+I.
\end{align*}
We will prove that the right hand side term is nonpositive. 
 Observe that $I$ can be written as follows
\[
  I=\int _{\rr}v(t,y)\int _{\rr }\sgn(v^+(t,x))\Big[K(x-y)+mG(x-y)+G(x-y)\frac {v(t,y)}4+G(x-y)\frac{v(t,x)}2\Big]dxdy.
\]
Using the assumption on the $L^\infty$-norm of the initial data we get
$-M-m\leq v(t)\leq M-m$ for all $t\in I$ and hence 
\[
\frac m4-\frac {3M}4 \leq  m+ \frac {v(t,y)}4+\frac{v(t,x)}2\leq \frac {3M}4+\frac m4.
\]
Under the assumption that  $m<1/\cgk$ we always can choose $M\leq 1/\cgk$. Thus  we obtain that 
\begin{align*}
\label{}
|G(x-y)|\Big|m+ \frac {v(t,y)}4+\frac{v(t,x)}2 \Big|\leq \cgk K(x-y) ( \frac {3M}4+\frac m4)  \leq K(x-y).
\end{align*}
 It follows that the bracket term in $I$ is nonnegative and 
 $I$ satisfies
\[
  I\leq \iint _{\rr^2}v^+(t,y)\Big[K(x-y)+mG(x-y)+G(x-y)\frac {v^+(t,y)}4+G(x-y)\frac{v^+(t,x)}2\Big]dxdy.
\]
 Since $G$ is an odd function we have for any function $w$ that
\begin{equation}
\label{asimetry}
  \iint _{\rr^2}w(t,y)G(x-y)w(t,x) dxdy=0.
\end{equation}
%
%
We apply this identity with $w=v^+$ and 
using that $K$ has mass one and $G$ has mass zero we obtain
\begin{align*}
 I
 &\leq\int _{\rr}v^+(t,y) \int_{\rr}\Big[K(x-y)+mG(x-y)+G(x-y)\frac{v^+(t,y)}4\Big]dxdy=\int _{\rr}v^+(t,y)dy.
\end{align*}
 This implies that
\[
  \frac {d}{dt}\int _{\rr} v^+(t,x)dx\leq 0
\]
 and $u(t)\leq m$ in $I$. A similar  strategy  works to prove that $u(t)\geq -m$ by denoting $w=-u$, writing the equation for $w$ and applying the same arguments with $G$ replaced by $-G$.

\medskip

\textbf{Step II. Control of the $L^1(\rr)$-norm.}
We first prove that for any interval $I$ such that
  $\|u(t)\|_{L^\infty}\leq m<1/\cgk$ for all $t\in I$ we have
\begin{equation}\label{}
\frac {d}{dt}\int _{\rr}|u(t,x)|dx\leq 0, t\in I.
\end{equation}
Using  equation \eqref{conv-diff-1}  we obtain that
\begin{align}
\label{term.cu.A}
\nonumber \frac{d}{dt}\int _{\rr}|u(t,x)|dx&=-\int _{\rr} |u(t,x)|dx +\int _{\rr}\Big(K\ast u+G\ast \frac{u^2}4 + (G\ast  u )\frac u2 \Big) \sgn(u(t,x))dx\\
&=-\int _{\rr} |u(t,x)|dx +A,
\end{align}
where $A$ is given by
\[
A= \int_{\rr} u(t,y) \int_{\rr} \sgn(u(t,x)) \Big[K(x-y)+\frac  {u(t,y)}4 G(x-y) +\frac { u(t,x)}2G (x-y)\Big]dxdy.
\]
Observe now that under the assumption $\|u(t)\|_{L^\infty}\leq m$ we have
\[
  |G(x-y)|\Big|\frac  {u(t,y)}4  +\frac { u(t,x)}2\Big|\leq\frac {3m}4 \cgk K(x-y)| \leq K(x-y).
\]
This implies that the terms of the above integral are negative unless $u(t,y)$ and $u(t,x)$ have the same sign. We write $A\leq A^+-A^-$
where
\begin{align*}
A^{\pm}= \iint _{\rr^2}  u^\pm(t,y)\Big[ K(x-y)+\frac  {u^\pm(t,y)}4 G(x-y) +\frac { u^\pm(t,x)}2 G (x-y)\Big]dxdy.
\end{align*}
Hence using identity \eqref{asimetry} with $w=u^{\pm}$ and the fact that $K$ and $G$ have mass one and zero respectively,
we get
\[
  A^\pm=\int _{\rr} u^{\pm}(t,y)dy.
\]
This shows that the right hand side in \eqref{term.cu.A} is negative and then 
  $\|u(t)\|_{L^1(\rr)}\leq \|\varphi\|_{L^1(\rr)}$ for all $t\in I$. Repeating the same argument it follows that the same holds for all $t>0$.

It remains to prove that the solutions keep the signum of the initial data. It is sufficient to consider the case when $\varphi\leq 0$. In this case we can choose $m=0$ in \eqref{eq.v} and repeat the same arguments as in Step I. 
The proof is now finished. 
\end{proof}

\medskip

Before proving the decay of the solutions we need the following comparison principle for the solutions of equation \eqref{conv-diff}.
\begin{theorem}\label{max-principle}
For any initial data $u_0,v_0\in L^1(\rr)\cap L^\infty(\rr)$ satisfying 
\begin{equation}\label{asump.infty}
\|u_0\|_{L^\infty(\rr)}< 1/\cgk, \ \|v_0\|_{L^\infty(\rr)}< 1/\cgk, \ u_0\leq v_0,
\end{equation}
the corresponding solutions of equation \eqref{conv-diff} satisfy the comparison principle
$$u(t)\leq v(t), \quad \forall\, t\geq 0.$$
\end{theorem}

%
%

\begin{proof}
Let us consider $u$ and $v$  solutions of \eqref{conv-diff} with initial data $u_0$ and $v_0$.
 In view of Theorem \ref{global} the solutions are global and the $L^\infty(\rr)$-norm  of each one is  less than $1/\cgk$.
We denote $w=u-v$. It follows that $w$ satisfies the equation
\begin{equation}\label{difference}
\left\{
\begin{array}{ll}
w_t=K\ast w-w+ \frac14  G\ast [(u+v)w] +\frac 12 w (G\ast u)+\frac 12v(G\ast w) ,t>0,\\[10pt]
w(0)=u_0-v_0.
\end{array}
\right.
\end{equation}	
Let us choose a positive constant $\alpha$ that will be chosen later. We have the following
\begin{equation}
\label{}
  \frac {d}{dt}\Big(e^{-\alpha t} \int _{\rr} w^+(t,x)dx\Big)=e^{-\alpha t}\Big( -\alpha\int _{\rr}w^+(t,x)dx +\int _{\rr} w_t(t,x)\sgn (w^+(t,x))dx \Big).
\end{equation}
We will prove that for  $\alpha$ large enough, the term in the right hand side  of the above integral is negative and the proof is finished.
Let us remark that 
\[
  \int_\rr (G\ast u) w^+(t,x)dx\leq \|G\|_{L^1(\rr)}\|u\|_{L^\infty(\rr)}\int _{\rr}w^{+}(t,x)dx.
\]
It is then sufficient to estimate the following term
\begin{align*}
\label{}
I&=\int _{\rr} \Big(K\ast w +\frac14  G\ast [(u+v)w]  +\frac 12v(G\ast w)\Big) \sgn (w^+)dx\\
&= \int _{\rr} \int _{\rr}w(y)  \sgn (w^+(x))  \Big(K(x-y)+G(x-y)\frac {u(y)+v(y)}4+G(x-y)\frac{v(x)}2\Big)dxdy.
\end{align*}
We remark that the big term in the integral is nonnegative since
\begin{align*}
\label{}
  |G(x-y)| \Big|\frac{u(y)+v(y)}4+\frac{v(x)}2\Big| &\leq |G(x-y)| \max \{\|u\|_{L^\infty(\rr)}, \|v\|_{L^\infty(\rr)}\}\\
  &\leq |G(x-y)|/ \cgk\leq K(x-y).  
\end{align*}

Thus
\begin{align*}
\label{}
  I&\leq \int _{\rr} \int _{\rr}w^+(y)   \Big(K(x-y)+G(x-y)\frac {u(y)+v(y)}4+G(x-y)\frac{v(x)}2\Big)dxdy\\
  &= \int _{\rr}w^+(y) +\frac 12 \int _{\rr}w^+(y) G(x-y){v(x)}dxdy\\
  &\leq  (1+\frac 12 \|G\|_{L^1(\rr)}\|v\|_{L^\infty(\rr)})\int _{\rr}w^+(y).
\end{align*}
Choosing $\alpha$ large enough   we obtain that $w^+\equiv 0$ so $u\leq v$. The proof is now finished.
\end{proof}

\begin{proof}[Proof of Theorem \ref{decay}.] We consider the case of nonnegative solutions since otherwise we consider $\tilde \varphi =\pm |\varphi|$ and by the comparison principle in Theorem \ref{max-principle} the corresponding solutions satisfy $|u|\leq |\tilde u|$ reducing the problem to the case on nonnegative solutions.

We multiply equation \eqref{conv-diff-1} by $u^{p-1}$, $p\geq2$, and obtain 
\begin{align}
\nonumber \frac 1p\frac{d}{dt}\int _{\rr} u^p(t,x)dx&=\int _{\rr} (K\ast u-u)u^{p-1}+ \frac 14\int _{\rr} (G\ast u^2)u^{p-1}+\frac 12 \int _{\rr} (G\ast u)u^p\\
\label{iden.202}&=-\frac{I_1}{2}+I_2,
\end{align}
where
\begin{equation}
\label{i1}
  I_1=\int _{\rr}\int _{\rr}K(x-y)(u(x)-u(y))(u^{p-1}(x)-u^{p-1}(y))dxdy
\end{equation}
and
\begin{equation}
\label{i2}
  I_2=\int _{\rr}\int _{\rr}G(x-y)\Big(\frac {u^2(y)u^{p-1}(x)}{4}+\frac {u(y)u^{p}(x)}{2}\Big)dxdy. 
\end{equation}

The idea of the proof is that under the smallness assumption on the $L^\infty(\rr)$-norm of the solution $u$ we can prove that $|I_2|\leq  C I_1$ for some constant $C<1/2$. Thus, the right hand side in \eqref{iden.202} can be commpared with $-I_1$. Once this is proven we can follow as in \cite[Section 5]{MR2356418} to prove the decay of the solutions.
 
\begin{lemma}\label{i2i1}
	For any integer $p\geq 2$ the following inequality holds
	$$
  I_2\leq C(p) \cgk I_1,$$
where $C(p)$ is given by 
$$C(p)=
\left\{
\begin{array}{ll}
\frac 14,& p=2,\\[10pt]
	\frac{p}{p+1},& p\geq 3.
\end{array}
\right.
$$
\end{lemma}

%
%
%

Using the above Lemma we obtain
 \begin{align*}
\frac 1p\frac{d}{dt}\int _{\rr} u^p(t,x)dx&\leq -\Big(\frac 12-C(p)\cgk\|u_0\|_{L^\infty(\rr)}\Big)I_1
\end{align*}
When $p=2$ it follows that under the assumption that   $\|u_0\|_{L^\infty(\rr)}\leq 1/\cgk$ the solution satisfies 
\begin{equation}
\label{h1}
\frac 12\frac{d}{dt}\int _{\rr} u^2(t,x)dx\leq -\frac 14 \int_{\rr} \int_{\rr}K(x-y)( u(x) - u(y))^2dxdy .
\end{equation}
When $p\geq 3$ we need to assume that $\|u_0\|_{L^\infty(\rr)}\leq 1/2\cgk$ in order to obtain that
\[
\frac 1p\frac{d}{dt}\int _{\rr} u^p(t,x)dx\leq -CI_1 \]
for some positive constant $C=C(p,\cgk,\|u_0\|_{L^\infty(\rr)})$.

Proceeding now as in \cite[Section 5]{MR2356418} we obtain the desired decay property of the solutions for any $p\in [2,\infty)$. When $p\in [1,2)$ we obtain the same result by  interpolating the case $p=2$ with the estimate on the $L^1(\rr)$-norm of the solutions obtained in Theorem \ref{global}.
The proof is now finished.
\end{proof}

\begin{proof}[Proof of Lemma \ref{i2i1}]
We first consider the case $p=2$. In this case, the structure of $I_2$ and the fact that $G$ is an odd function show that after a change of variables we have 
\[
  I_2=\frac 14\int _{\rr}\int _{\rr}G(x-y) {u(y)u^{2}(x)}dxdy=  \int _{\rr}\int _{\rr}G(x-y)\Big(\alpha u^{3}(y)+\frac{u(y)u^{2}(x)}{4}+\beta u^{3}(x)\Big)dxdy, 
\]
for any real numbers $\alpha$ and $\beta$.
We now choose $\alpha$ and $\beta$ such that  polynomial $P(z)=\alpha z^3+z/4+\beta$ has a zero of second order at $z=1$. This means  $\alpha=-1/12$
and $\beta=-1/6$ and $P(z)=(z-1)^2(z-2)/12$. Thus
\[
  |P(z)|\leq \frac 14(z-1)^2 \max \{z,1\}, \quad \forall z\geq 0. 
\]
Using this inequality we have that $I_2$ satisfies
\begin{align*}
\label{}
  |I_2|&\leq \frac{\cgk}4  \int _{\rr}\int _{\rr} K(x-y)(u(x)-u(y))^2 \max\{u(x),u(y)\}dxdy\leq  \frac{\cgk}4  \|u\|_{L^\infty(\rr)} I_1.
\end{align*}

 Let us now consider the case $p\geq 3$. In this case we cannot use the symmetry of the two terms that appear inside the integral in $I_2$ and we have to take care of both of them. 
Using  that 
$G$ is an odd function  for any real numbers $\alpha$ and $\beta$ we can write $I_2$ as
\begin{equation}
\label{i22}
  I_2=\int _{\rr}\int _{\rr}G(x-y)\Big(\alpha u^{p+1}(y)+\frac{u^2(y)u^{p-1}(x)}{4}+\frac{u(y)u^{p}(x)}{2}+\beta u^{p+1}(x)\Big)dxdy. 
\end{equation}

We are looking for $\alpha$ and $\beta$ such that for some constant $C(p)$ the following inequality holds 
\begin{equation}\label{ineg.1}
\Big |\alpha z^{p+1} + \frac {z^2}4+\frac z2+\beta \Big|\leq C(p) (z-1)(z^{p-1}-1)\max\{z,1\}, \quad \forall \, z>0.
\end{equation}
 Let us choose $\alpha$ and $\beta$ in such a way that 
$f(z)=\alpha z^{p+1} + \frac {z^2}4+\frac z2+\beta$
has a zero of second order at $z=1$: $f(1)=f'(1)=0$, i.e. $\alpha +\beta=-3/4$, $(p+1)\alpha +1=0$. It gives us that $\alpha=-1/(p+1)$ and $\beta=-3/4+\frac 1{p+1}$. Back to inequality \eqref{ineg.1} we claim that we can choose  $C=p/(p+1)$.
Using   inequality \eqref{ineg.1} we find that
\begin{align}\label{I2}
|I_2|&\leq  \frac{p}{p+1} \int_{\rr}\int_{\rr} |G(x-y)|( u(y) - u(y))(u^{p-1}(x)-u^{p-1}(y))\max\{ u(x),u(y)\}dxdy\\
\nonumber &\leq \frac{p}{p+1} \cgk  \|u_0\|_{L^\infty(\rr)} I_1.
\end{align}

It remains to prove inequality \eqref{ineg.1}.  Since both terms in \eqref{ineg.1} have $z=1$ as root with multiplicity two we can divide them by $(z-1)^2$ and after explicit computations it remains to prove that
\[
Q(z)=z^{p-1}+2z^{p-2}+\dots +(p-1)z+\frac{3p-1}4\leq (p-1) R(z)\max\{z,1\}, \quad\forall z\geq 0. 
\]
where $R(z)=(z^{p-2}+\dots +z+1)$.
Since $p\geq 3$ we have that $(3p-1)/4\leq p-1$ and 
\begin{align*}
\label{}
Q(z)&=zR(z)+z^{p-2}+\dots+(p-2)z+\frac{3p-1}4 \leq z R(z)+ z^{p-2}+\dots+(p-2)z+p-1   \\
&\leq zR(z)+ (p-1)R(z).
\end{align*}

The proof is now complete.
\end{proof}

\section{Asymptotic expansion of the solutions}\label{expansion}

In this section we prove Theorem \ref{asimp}.
We introduce the family of rescaled solutions
 $u_\lambda(x,t)=\lambda u(\lambda x,\lambda^2 t).$ 
 Each $u_\lambda$ satisfies the following equation
\begin{equation}\label{rescal}
\left\{
\begin{array}{ll}
u_{\lambda,t}=\lambda^2(K_\lambda\ast u_\lambda-u_\lambda)+\lambda  \int _{\rr}G_\la (x-y)\Big(\frac {u_\la(t,x)+u_\la(t,y)}2\Big)^2,& t>0,x\in \rr,\\[10pt]
u_\lambda(0)=\varphi_\lambda(x), &x\in \rr.
\end{array}
\right.
\end{equation}
where $K_\lambda$ and $G_\lambda$ are the rescaled kernels:
\begin{equation}\label{kernels}
K_{\lambda}(x)=\lambda K(\la x), \quad G_\la(x)=\la G(\la x).
\end{equation}

%
%
%

We will divide the proof of Theorem \ref{asimp} in various steps. 
We first obtain uniform estimates for our rescaled solutions. They are going in the same direction as the ones in \cite{denisa} but we have to take care of the term $T$ introduced in the first section. This term has no sign property and should be absorbed in the diffusive part.

{\bf Step I. Uniform estimates on the rescaled solutions}. In the following Lemmas we obtain  uniform estimates on the family $\{u_\lambda\}_{\lambda>0}$.
\begin{lemma}\label{est.2.compac}
There exists a positive constant  $C=C(\|\varphi\|_{L^1(\rr)},\|\varphi\|_{L^\infty(\rr)})$ such that 
 the following hold uniformly on $\lambda>0$:
\begin{equation}
\label{est.norme.p}
  \|u_\lambda(t)\|_{L^p(\rr)}\leq C t^{-\frac{1}{2}(1-\frac1p)}, \quad \forall 1\leq p<\infty, \, \forall t>0,
\end{equation}
\begin{equation}\label{a1,a2}
\lambda^2 \int_{t_ 1}^{t_2}\int_\rr\int_\rr K_\lambda(x-y)(u_\lambda(t,y)-u_\lambda(t,x))^2dydxdt\leq C t_1^{-\frac 12}, \quad \forall \, 0<t_1<t_2<\infty.
\end{equation}
\end{lemma}
\begin{proof}
The proof of the first part uses the definition of the rescaled family $u_\lambda$ and estimate \eqref{decay.p}.
Using identity \eqref{h1}  we obtain that for any $0<t_1<t_2<\infty$ the following holds:
\begin{equation}\label{t1,t2}
\int^{t_ 2}_{t_1}\int_\rr\int_\rr K(x-y)(u(t,y)-u(t,x))^2dydxdt\lesssim  \|u(t_1)\|_{L^2(\rr)}^2.
\end{equation}
Let us now consider the same quantity for the rescaled solution. 
After a change of variables and  using estimate \eqref{decay.2} we have that
\begin{align*}
\lambda^2&\int^{t_ 2}_{t_1}\int_\rr\int_\rr K_\lambda(x-y)(u_\lambda(t,y)-u_\lambda(t,x))^2dydxdt\\
&=\lambda \int_{\la ^2 t_ 1}^{\la ^2 t_2}\int_\rr\int_\rr K(x-y)(u(t,y)-u(t,x))^2dydxdt\leq \la \|u(\la^2 t_1)\|_{L^2(\rr)}^2\leq C t_1^{-1/2}.
\end{align*}
The proof is finished.
\end{proof}

\begin{lemma}\label{energy}
For any $0<t_1<t_2<\infty$  there exists a positive constant $M(t_1)$ such that  for all $\lambda>1$ the following holds
\begin{equation}\label{normL2}
\|u_{\lambda,t}\|_{L^2((t_1,t_2),H^{-1}(\rr))}\leq M(t_1).
\end{equation}
\end{lemma}

\begin{proof}

Let be a function $\Phi\in C^2_c(\rr)$ with $\|\Phi\|_{H^1(\rr)}\leq 1$. Multiplying  equation \eqref{rescal} with $\Phi$, integrating over $\rr$ and making a change of variables we obtain:
\begin{align*}
\int_\rr u_{\lambda,t}(x,t)\Phi(x)dx
&=-\frac{ \lambda^2}2\int_\rr\int_\rr  K_\lambda(x-y) (u_\lambda(t,y)-u_\lambda(t,x)(\Phi(y)-\Phi(x))dxdy\\
& \quad +\frac\lambda2\int_\rr   \int _{\rr}G_\la (x-y)\Big(\frac {u_\la(t,x)+u_\la(t,y)}2\Big)^2(\Phi(x)-\Phi(y))dxdy\\
&=\frac {A_\lambda}2+\frac {B_\lambda}8.
\end{align*}

We now estimate the two terms $A_\lambda$ and $B_\lambda$.
 Let us recall that for any $K\in L^1(\rr,|x|^2)$ and $v\in H^1(\rr)$  the following estimate (see \cite[Lemma 2.3]{denisa}) holds for all $\lambda>0$
\begin{equation}\label{ineg-1}
\lambda^2\int _{\rr} \int _{\rr} K_\lambda (x-y)(v(x)-v(y))^2dxdy\leq \int _{\rr} K(x)|x|^2dx \int _{\rr} | v'(x)|^2dx.
\end{equation}

Using
Cauchy's inequality and inequality \eqref{ineg-1} we obtain:
\begin{align*}
|A_\lambda|^2&\leq\left[ {\lambda ^2}\int_\rr\int_\rr K_\lambda(x-y)|\Phi(y)-\Phi(x)|^2dxdy\right]\left[ {\lambda^2} \int_\rr\int_\rr K_\lambda(x-y)|u_\lambda(t,y)-u_\lambda(t,x)|^2dxdy\right]\\
&\lesssim  \|\Phi'\|_{L^2(\rr)} {\lambda^2} \int_\rr\int_\rr K_\lambda(x-y)|u_\lambda(t,y)-u_\lambda(t,x)|^2dxdy.
\end{align*}
We now consider the term $B_\lambda.$  We  use that $|G|\leq \cgk K$, Cauchy's inequality and the decay of the $L^4(\rr)$-norm of $u_\lambda$ in \eqref{est.norme.p} to obtain
\begin{align*}
  |B_\lambda|^2&\lesssim \left[\lambda ^2 \int_\rr\int_\rr K_\lambda(x-y)|\Phi(y)-\Phi(x)|^2dxdy\right]\left[ \int_\rr\int_\rr K_\lambda(x-y)|u_\lambda(t,y)-u_\lambda(t,x)|^4dydx\right]\\
  &\lesssim  \|\Phi'\|_{L^2(\rr)}\int_\rr\int_\rr K_\lambda(x-y)(u_\lambda^4(t,y)+u_\lambda^4(t,x))dxdy\\
  &\lesssim   \|\Phi'\|_{L^2(\rr)} \|u_\la(t)\|_{L^4(\rr)}^4\leq t_1^{-3/2}   \|\Phi'\|_{L^2(\rr)} .
\end{align*}
The two estimates on $A_\lambda$ and $B_\lambda$ show that 
\[
  \|u_{\lambda,t}(t)\|_{H^1(\rr)}^2\lesssim {\lambda^2} \int_\rr\int_\rr K_\lambda(x-y)|u_\lambda(t,y)-u_\lambda(t,x)|^2dydx+t_1^{-3/2}.
\]
Integrating the above inequality over $[t_1,t_2]$ and  using \eqref{a1,a2} we obtain the desired estimate.

\end{proof}

{\bf Step II. Compactness of the trajectories $\{u_\lambda\}$  in $L^1_{loc}((0,\infty),L^1(\rr))$.}
Using estimates \eqref{est.norme.p} and \eqref{normL2} we obtain that there exists a function $U\in L^\infty_{loc}((0,\infty), L^p(\rr))$ such that, up to a subsequence, $u_\lambda(t)\rightharpoonup U(t)$
in $L^p_{loc}(\rr)$ for all $t>0$ and for all $1<p<\infty$. For the full details we refer to \cite[Section 2.2, Step II]{denisa}.
Estimate \eqref{est.norme.p} transfers to $U$:
  \begin{equation}
\label{est.U}
\|U(t)\|_{L^p(\rr)}\leq C t^{-\frac 12(1-\frac 1p)}, \quad \forall t>0, \ \forall 1<p<\infty.  
\end{equation}

Let us now prove the  convergence of $u_\lambda$ toward $U$ in 
$L^2_{loc}((0,\infty)\times \rr)$. We will use the following compactness argument for nonlocal evolution problems that has been proved in \cite{MR3190994} (see also \cite{denisa} for a more general result). 
\begin{theorem}(\cite[Th. 5.1]{denisa})\label{compac}
Let $\rho:\rr\rightarrow\rr$ be a nonnegative smooth even function with compact support, such that $\rho(x)\geq \rho(y)$ for all $|x|\leq |y|$ and $\rho_n(x)=n\rho(nx)$. 
Assume that
 $\{f_n\}_{n\geq1}$ is a sequence in $L^2((0,T), L^2(\rr))$ such that for some positive constant $M$ the following hold:\\
1) $ \|f_n\|_{L^2((0,T)\times\rr)}\leq M,$\\
2) $n^2\int^T_0\int_{\rr}\int_{\rr} \rho_n(x-y)(f_n(y,t)-f_n(x,t))^2dydxdt\leq M$, for all $n\geq 1$\\
and\\
3) $\|\partial_t f_n (t)\|_{L^2((0,T);H^{-1}(\rr))}\leq M$ for all $n\geq 1$.\\
Then there exists a function  $f\in L^2((0,T);H^1(\rr))$ such that, up to a subsequence, $f_n\to f$  in $L^2_{loc}((0,T)\times\rr).$
\end{theorem}

We use now that our function $K$ is positive in a neighborhood of the origin. Thus we can take a function $\rho\leq K$ satisfying the properties in Theorem \ref{compac}.
Applying this result to the family $\{u_\lambda\}_{\lambda\geq 0}$ in the  time interval $(t_1,t_2)$ we obtain that there exists a function $v\in L^2((t_1,t_2),H^1(\rr))$ such that up to a subsequence $u_\lambda\to v$ in $L^2((t_1,t_2);L^2_{loc}(\rr)).$  Using the uniqueness of the limit we obtain that in fact $v=U$.
Thus $U\in L^2_{loc}((0,\infty);H^1(\rr))\cap L^1_{loc}((0,\infty)\times\rr)$ and $u_\lambda\to U$ in  $L^2_{loc}((0,\infty)\times \rr)$ so the same holds in $L^1_{loc}((0,\infty)\times \rr)).$ 

Now we prove that $u_\lambda$ converges strongly  to $v$ in $L^1_{loc}((0,\infty);L^1(\rr))$ by obtaining uniform estimates on the tails of 
$\{u_\la\}_{\la>1}$.
\begin{lemma}\label{strong}
There exists a constant $C=C(K,\|\varphi\|_{L^1(\rr)},\|\varphi\|_{L^\infty(\rr)})$ such that the following inequality 
\begin{equation}\label{ineq}
\int_{|x|>2R}u_\lambda (t,x)dx\leq\int_{|x|>R}\varphi(x)dx+C\left(\frac t{R^2}+\frac{t^\frac 12}{R}\right)
\end{equation}
holds for any $t>0$ and $R>0$, uniformly on $\lambda\geq 1.$
\end{lemma}

\begin{proof}
We consider the function $\Psi\in C^\infty(\rr)$, defined by 
\begin{equation}\label{psi}
\Psi(x)=\left\{
\begin{array}{ll}
0, & |x|<1,\\
1, & |x|>2.
\end{array}
\right.
\end{equation}
We set $\Psi_R(x)=\Psi(x/R)$.
We multiply  equation \eqref{rescal} with $\Psi_R(x)$ and integrate in time and space
\begin{align*}
\int_\rr  & u_\lambda(x,t)\Psi_R(x)dx-\int_\rr \varphi_\lambda(x)\Psi_R(x)(x)dx\\
&=\int^t_0\int_\rr\lambda^2((K_\lambda\ast \Psi_R)(x)-\Psi_R(x))(x)u_\lambda (s,x)dxds\\
\nonumber &\quad +\frac{\lambda}{8}\int^t_0\int_\rr \int_\rr G_\lambda (x-y) (u_\lambda(s,x)+u_\lambda(x,y))^2(\Psi_R(x)-\Psi_R(y))dxds\\
\nonumber&=A_\la+B_\la/8.
\end{align*}
In the case of $A_\la$ we use the results in \cite[Lemma 2.2]{denisa}  and the $L^1(\rr)$-estimate on $u$, so on $u_\lambda$,
to obtain
\begin{align*}
|A_\la|&\leq\int^t_0\|\lambda^2(K_\lambda\ast\Psi_R-\Psi_R)(x)\|_{L^\infty(\rr)}\|u_\lambda(s)\|_{L^1(\rr)}ds\\
&\lesssim t\|(\Psi_R)_{xx}\|_{L^\infty(\rr)}\|\varphi\|_{L^1(\rr)}\lesssim \frac t{R^2}\|\varphi\|_{L^1(\rr)}.
\end{align*}
The second term, $B_\la$, satisfies
\begin{align*}
|B_\lambda|&\lesssim  \lambda \|(\psi_R)' \| _{L^\infty(\rr)}\int _0^t \iint_{\rr^2}|G_\lambda(x-y)| (u^2_\lambda(x,y)+u^2_\lambda(s,x))|x-y|dxdyds\\
&\lesssim \frac\lambda R\int _0^t \iint _{\rr^2}u^2_\lambda(s,y) |G_\lambda(x-y)| |x-y|dxdyds\\
&=\frac 1R \int _{\rr}G(z)|z|dz \int _0^t\int _{\rr} u_{\lambda}^2(s,x)dxds\lesssim \frac 1R \int _0^t s^{-1/2}=\frac{t^{1/2}}R.
\end{align*}
Hence
$$\int_{|x|>2R}u_\lambda (t,x)dx\leq \int_{\rr}u_\lambda(t,x) \psi_R(x)dx\leq \int_\rr{\varphi_\lambda}(x)\psi_R(x)+C(\frac t{R^2}+\frac{t^{1/2}}R)$$
and the proof finishes.
\end{proof}

\textbf{Step III. Passing to the limit.}
Using the results obtained in the previous step we can pass to the limit in the equation involving $u_\lambda$.
Let us choose $0<\tau<t$. For any test function $\psi\in C^\infty_c(\rr^d)$, we multiply   \eqref{rescal} by $\psi$ and we integrate over $(\tau,t)\times\rr$. We get: 
\begin{align}\label{Psi_x}
\int^t_\tau\int_\rr u_{\lambda,s}(s,x)&\psi(x)dxds =\int^t_\tau\int_\rr \lambda^2(K_\lambda\ast u_\lambda-u_\lambda)(s,x)\psi(x)dxds\\
\nonumber &\quad +\frac\lambda 4\int_\tau ^t \int _\rr \int_\rr G_\lambda(x-y) (u_\lambda(s,x)+u_\lambda(s,y))^2 \psi(x)dxdyds.
\end{align}
Integrating by parts with respect to variable $s$ in the left hand side and changing the variables in the right hand side 
we obtain:
\begin{align*}
\int_\rr u_{\lambda}(t,x)\psi(x)dx-&\int_\rr u_{\lambda}(\tau,x)\psi(x)dx =\int^t_\tau\int_\rr\lambda^2(K_\lambda\ast\psi-\psi)(x)u_\lambda(s,x)dxds\\
\nonumber &\quad +\frac\lambda 8\int_\tau ^t \int _\rr \int_\rr G_\lambda(x-y) (u_\lambda(s,x)+u_\lambda(s,y))^2 (\psi(x)-\psi(y))dxdyds.
\end{align*}
Since for $t>0$, $u_\lambda(t)\rightharpoonup  U(t)$ in $L^2_{loc}(\rr)$
we can pass to the limit in the left hand side term
\begin{align}\label{conv_1}
\int_\rr  u_{\lambda}(t,x)\Psi(x)dx-\int_\rr u_{\lambda}(\tau,x)\Psi(x)dx \to \int_\rr U(t,x)\Psi(x)dx-\int_\rr U(\tau,x)\Psi(x)dx.
\end{align}
Using the same arguments as in \cite{denisa} we have 
\begin{align}\label{conv_2}
A_\lambda \to \int^t_\tau\int_\rr U(s,x)A\Psi_{xx}(x)dxds.
\end{align}
We now prove that 
\begin{align}
\label{ultimultermen}
B_\lambda\to  4B \int_\tau ^t \int _\rr U^2(s,x)\psi'(x)dxds, \quad \lambda\to\infty.
\end{align}

We first point out that since $u_\lambda\rightarrow U$ in $L^1((\tau,T)\times \rr)$ and $u_\lambda(s)$, $s\in [\tau,t]$ is uniformly bounded in any space $L^p(\rr)$ 
with $p>2$ we obtain that $u_\lambda\rightarrow U$ in $L^2((\tau,T)\times \rr)$.
Making a change of variables and using a Taylor expansion of order one 
 we write $B_\lambda$ as follows
\begin{align*}
  B_\lambda &=\int _{\rr} G(z)z \int _{\tau}^t \int _{\rr}   \Big( u_\lambda(y+\frac{z}{\lambda})+u_\lambda(y) \Big)^2  \int _{0}^1 \psi'\Big(y+\frac{\sigma z}{\lambda}\Big)
d\sigma dy ds\\
  &=\int _{\rr} G(z)z I_\lambda(z)dz.
\end{align*}
Note that $I_\lambda(z)$ satisfies 
\[
  I_\lambda(z)\lesssim \|\psi'\|_{L^\infty(\rr)} \int _\tau ^t  \|u_\lambda (s)\|_{L^2(\rr)}^2ds\leq C(\tau) \|\psi'\|_{L^\infty(\rr)}.
\]
We claim that for any $z\in \rr$ the following holds
\begin{equation}
\label{claim.1}
   I_\lambda(z)\to 4B \int_\tau ^t \int _\rr U^2(s,x)\psi'(x)dxds,\quad \text{as}\
   \lambda\rightarrow \infty.
\end{equation}
Using  that $G$ has a first momentum in $L^1(\rr)$ and  Lebesque's dominated convergence theorem we obtain the desired estimate.

We now prove claim \eqref{claim.1}. Since $u_\lambda\to  U$ in $L^2((\tau,t)\times \rr)$ then $u_\lambda(\cdot, \cdot +z/\lambda)\rightarrow U$ in $L^2((\tau,t)\times \rr)$. This gives us that 
$( u_\lambda(\cdot +{z}/{\lambda})+u_\lambda )^2\rightarrow 4U^2$ in $L^1((\tau,t)\times \rr)$.
Moreover $\int _0^1\psi'(y +\sigma z/\lambda)d\sigma \to \psi'(y) $  uniformly in $y \in \rr$. This proves \eqref{claim.1}.

Using \eqref{conv_1}, \eqref{conv_2} and \eqref{ultimultermen} we obtain that $U$  satisfies
\[
  U_t=AU_{xx}-B(u^2/2)_x,\quad \text{in}\, \mathcal{D}'((0,\infty)\times\rr).
\]

In order to identify the initial data with the same arguments as in Lemma \ref{strong} we obtain that for any $\psi\in C_c^2(\rr)$ the following holds:
\[
 \Big  |\int _\rr u_\lambda(t,x)\psi(x)dx-\int _\rr u_\lambda(0,x)\psi(x)dx\Big|\lesssim t\|\psi''\|_{L^\infty(\rr)}+t^{1/2}\|\psi'\|_{L^\infty(\rr)}.
\]
This shows that 
\[
  \lim_{t\downarrow 0}\int _\rr U(t,x)\psi(x)dx=m\psi(0).
\]
By an approximation argument we obtain that $U(t)\rightarrow m\delta_0$ in the sense of bounded measures on $\rr$.

Using that $U\in  L^2_{loc}((0,\infty),H^1(\rr))$, estimate \eqref{est.U} and classical regularity results for parabolic equations we obtain that $U$
is the unique solution of system \eqref{burgers} (see \cite{MR1032626} for more details about the existence and uniqueness of solutions for system \eqref{burgers}). This guarantees that the whole family $\{u_\lambda\}_{\lambda>0}$ converge to $U$, and  not only a subsequence.

{\bf Step IV. } From Step II it follows that for some $t_0>0$ we have
\[
  \|u_\lambda(t_0)-U(t_0)\|_{L^1(\rr)}\rightarrow 0\quad \text{as}\quad \lambda\rightarrow \infty.
\]
This implies that
\[
  \lim _{t\rightarrow\infty}\|u(t)-U(t)\|_{L^1(\rr)}=0.
\]
When $1<p<\infty$ the decay property of the $L^p(\rr)$-norms, $1\leq p<\infty$, of $u_\lambda$ and $U$ gives us that
\[
  \|u(t)-U(t)\|_{L^p(\rr)}=o(t^{-\frac 12(1-\frac 1p)}),
\]
which finishes the proof.

\medskip
 {\bf
Acknowledgements.}

L. Ignat was partially supported by Grant PN-II-RU-TE- 2014-4-0007 of the Romanian National Authority for Scientific Research, CNCS -- UEFISCDI.

T. Ignat was  supported by Grant PN-II-RU-TE-2014-4-0657 of the Romanian National Authority for Scientific Research, CNCS -- UEFISCDI and by a doctoral fellowship offered by IMAR.


\bibliographystyle{plain}
\bibliography{biblio} 

\end{document}